\documentclass{amsart}
\usepackage{amsmath, amsfonts, amssymb,amsbsy,eucal,mathrsfs}
\usepackage[all]{xy}
\usepackage{pstricks}
\usepackage{pst-plot}
\usepackage{graphicx}
\usepackage{enumerate}
\usepackage[utf8]{inputenc}
\usepackage[T1]{fontenc}
\usepackage{tikz}

\newtheorem{definition}{Definition}[section]
\newenvironment{defi}{\begin{definition} \rm}{\end{definition}}

\newtheorem{prop}[definition]{Proposition}

\newtheorem{lemm}[definition]{Lemma}

\newtheorem{coro}[definition]{Corollary}
\newtheorem{theo}[definition]{Theorem}
\newtheorem{theo1}{Theorem}
\newtheorem{notation}[definition]{Notation}

\newtheorem{construction}[definition]{Construction}

\newtheorem{remark}[definition]{Remark}
\newenvironment{rema}{\begin{remark} \rm}{\end{remark}}
\newtheorem{remarks}[definition]{Remarks}

\newtheorem{example}[definition]{Example}

\newtheorem{examples}[definition]{Examples}

\newtheorem{nothing}[definition]{$\!\!$}

\newenvironment{proo}{{\flushleft \it Proof.}}{\hfill $\square$
  \vspace{2mm}}

\newenvironment{proo-prop}{{\flushleft \it Proof of Proposition
    \ref{prop-f_4}.}}{\hfill $\square$ \vspace{2mm}}
\newenvironment{proo-thm}{{\flushleft \it Proof of Theorem \ref{main-intro}.}}{\hfill $\square$ \vspace{2mm}}

\newtheorem{conjecture}[definition]{Conjecture}

\newtheorem{definition*}{Definition}[section]
\newenvironment{defi*}{\begin{definition*} \rm}{\end{definition*}}
\newtheorem{definitions*}[definition*]{Definitions}
\newenvironment{defis*}{\begin{definitions*} \rm}{\end{definitions*}}
\newtheorem{prop*}[definition*]{Proposition}
\newtheorem{lemm*}[definition*]{Lemma}
\newtheorem{coro*}[definition*]{Corollary}
\newtheorem{theo*}[definition*]{Theorem}
\newtheorem{remark*}[definition*]{Remark}
\newenvironment{rema*}{\begin{remark*} \rm}{\end{remark*}}
\newtheorem{remarks*}[definition*]{Remarks}
\newenvironment{remas*}{\begin{remarks*} \rm}{\end{remarks*}}
\newtheorem{example*}[definition*]{Example}
\newenvironment{exam*}{\begin{example*} \rm}{\end{example*}}
\newtheorem{examples*}[definition*]{Examples}
\newenvironment{exams*}{\begin{examples*} \rm}{\end{examples*}}
\newtheorem{nothing*}[definition*]{$\!\!$}
\newenvironment{noth*}{\begin{nothing*} \rm}{\end{nothing*}}

\newtheorem{commentaire*}[definition*]{Commentaire}

\def \rank {{{\rm rk}}}

\def \oo {{{\mathcal{O}}}}
\def \vs {\vskip}

\newcommand{\G}{\mathbb{G}}
\def \a {\alpha}
\def \b {\beta}

\newcommand{\p}{\mathbb{P}}

\newcommand{\scal}[1]{\langle #1 \rangle}

\begin{document}

\title{Spherical multiple flags}
\author{P. Achinger}
\address{Department of Mathematics, University of California, Berkeley CA 94720, USA}
\email{achinger@math.berkeley.edu}
\author{N. Perrin}
\address{Mathematisches Institut, Heinrich-Heine-Universit{\"a}t,
  40204 D{\"u}sseldorf, Germany}
\email{perrin@math.uni-duesseldorf.de}
\date{}

\begin{abstract}
For a reductive group $G$, the products of projective rational varieties homogeneous under $G$ that are spherical for $G$ have been classified by Stembridge. We consider the $B$-orbit closures in these spherical varieties and prove that under some mild restrictions they are normal, Cohen-Macaulay and have a rational resolution.
\end{abstract}

\maketitle

{\def\thefootnote{\relax}
 \footnote{ \hspace{-6.8mm}
 Mathematics Subject Classification:.}
 }

\vs - 0.7 cm

\centerline{\large{\bf Introduction}}

\vs 0.3 cm


A classical problem in geometric representation theory is to prove
regularity properties of $B$-orbit closures inside a $G$-variety $X$.
Here and henceforth $G$ is a reductive group over an algebraically
closed field $k$ and $B$ is a Borel subgroup of $G$. The most famous
example of such a theorem is the result of Mehta and Ramanathan
\cite{MR} that Schubert varieties (that is, $B$-orbit closures
inside $X = G/P$ a projective rational homogeneous space) are normal,
Cohen-Macaulay and have a rational resolution. For general
\emph{spherical varieties} (\emph{i.e.}, normal $G$-varieties with finitely many
$B$-orbits), this is more complicated and the $B$-orbit closures are
not even normal in general (for a survey of partial results in this direction, cf. \cite[Section 4.4]{survey}). In this paper, we restrict our attention to products
of homogeneous spaces. Our result is the following

\begin{theo1}
\label{main-intro} 
Assume that $G$ is a simply laced (i.e.,
with simple factors of types A, D, E). Let  $P_1, P_2$ be two
\emph{cominuscule} (see Definition \ref{def-comin}) parabolic subgroups of
$G$ containing $B$ and let $X = G/P_1 \times G/P_2$. Then the
$B$-orbit closures inside $X$ are normal, Cohen-Macaulay and have
a rational resolution.
\end{theo1}

To prove these regularity properties, we need to study in more detail
the $B$-orbit structure of $X$ and the \emph{weak order} (cf.
Definition \ref{defi-weak-order}) among the $B$-orbits. We prove the following two
facts, whose proof constitutes most of the paper, and which we hope
might be of independent interest:
\begin{enumerate}[(a)]
\item the minimal $B$-orbits with respect to the weak order are
$B\times B$-stable (see Theorem \ref{theo-spec}), hence their closures
are products of Schubert varieties, 
\item the action maps $P\times^B \oo \to P \oo$, where
$\mathcal{O}$ is a $B$-orbit in $X$, $P\supseteq B$ a minimal
parabolic subgroup with $P\mathcal{O} \neq \mathcal{O}$, are
birational (see Corollary \ref{N-coro}).
\end{enumerate}

With these results in hand, the structure of the proof of Theorem
\ref{main-intro} is as follows. For a $B$-orbit closure
$\bar \oo \subseteq X$, we find a minimal (with respect to the
weak order) $B$-orbit $\mathcal{O}'\preccurlyeq \mathcal{O}$. Since by (a)
the orbit closure $\bar \oo'$ is a product of two Schubert varieties, it admits a
rational resolution $Z\to \bar \oo'$ (for example, the product
of two Bott-Samelson resolutions \cite{demazure} of the two factors). Since $\oo' \preccurlyeq \oo$ there exists a sequence $P_{\gamma_1}, \cdots, P_{\gamma_i}$ of minimal parabolic subgroups of $G$ containing $B$ raising $\mathcal{O}'$ to $\mathcal{O}$ (see Definition \ref{defi-weak-order}). Then by (b) the action map
\[ P_{\gamma_r} \times^B \ldots \times^B P_{\gamma_1} \times Z \to \bar \oo  \]
is a rational resolution of $\bar \oo$. To prove normality, we
proceed by descending induction in the weak order. First, maximal
$B$-orbits are $G$-stable, and their closures are normal as locally trivial fibrations with Schubert varieties as fibers. In
the induction step, we use (a) and (b) again and follow ideas of Brion
\cite{brion2}. Finally, Cohen-Macaulayness follows from general arguments from Brion \cite[Section 3, Remark 2]{brion2} and the fact that this holds for the $G$-orbit closures.

In addition, we show that the "simply-lacedness" assumption in Theorem
\ref{main-intro} is necessary -- in Section 5 we find a $B$-orbit inside
$X=({\rm Sp}_6/P)^2$ where $P$ is a stabilizer of a $3$-dimensional
isotropic subspace whose closure is not normal (and in fact the
property (b) above fails). We do not know whether Theorem \ref{main-intro}
holds without the assumption that $P_1$ and $P_2$ be cominuscule. Note
that examples of such pairs with $X$ spherical are quite restricted --
see \cite{littelmann} and \cite{stembridge} for a complete list. The main reason
for the assumption that $P_1$, $P_2$ are cominuscule is that in
such case the $G$-orbits are induced from symmetric varieties (see Definition \ref{defi-ind} and Corollary \ref{cor-ind}) in which case minimal orbits for the weak order are closed.

The case of Theorem \ref{main-intro} when $X$ is a product of two Gra{\ss}mann
varieties was proved in \cite{BZ} thanks to
a detour into quiver representations. It was one of the motivations of
this work to present a direct proof of this
result. It was also inspired by a complete combinatorial description
of the weak order in a product of two Gra{\ss}mann varieties due to
Smirnov \cite{evgeny}, where the two phenomena (a) and (b) mentioned
above have been observed.

The structure of the paper is as follows. In Section 1, we define opposite pairs of parabolic subgroups and show how one can reduce the study of $G$-orbits inside $G/P_1 \times G/P_2$ to the case when $P_1$ and $P_2$ are opposite to each other. In that case the variety is symmetric which turns out to be very important. In Section 2, we recall the definition and basic properties of the weak order among the $B$-orbits in a spherical variety $X$ and prove (a) (Theorem \ref{theo-spec}). In Section 3, we introduce a distance function between torus-fixed points in $X$, generalizing a previous notion introduced in \cite{CMP1} and used in \cite{CMP2,CMP3,CP,BCMP} to study quantum cohomology of cominuscule rational homogeneous spaces. We use it to prove (b) (Corollary \ref{N-coro}). The proof of Theorem \ref{main-intro} occupies Section 4, and our counterexample with $G$ non-simply laced can be found in Section 5.

\tableofcontents

\section{Structure of $G$-orbits}
\label{section-g-orb}

Let $G$ be a reductive group, $T$ a maximal torus of $G$ and $B$ a Borel subgroup of $G$ containing $T$. Let $W=N_G(T)/T$ be the Weyl group associated to $T$. Let $P_1$ and $P_2$ be two parabolic subgroups of $G$ containing $B$ and define 
$$X=G/P_1\times G/P_2.$$ 
The variety $X$ has finitely many $G$-orbits. Any orbit is of the form: $G \cdot ( P_1 , w P_2)$ for some $w \in W$ and is isomorphic to $G/H$ with 
$$H=P_1\cap P_2^w$$ 
where $P_2^w=wP_2w^{-1}$. The inclusion morphism $\iota:G/H\to G/P_1\times G/P_2$ is induced by the morphism $G\to G\times G$ defined by $g\mapsto (g,g n_w)$ where $n_w$ is any representative of $w$ in $N_G(T)$.

In this section we prove a structure result on $G$-orbits which reduces the study to the case of an opposite pair $(P_1,P_2)$ (see Definition \ref{defi-opp-pairs}). For this we fix a $G$-orbit $G \cdot ( P_1 , w P_2) \simeq G/H$ of $X$ with $w \in W$ and $H=P_1\cap P_2^w$. 

\vs 0.2 cm

Recall that if $\chi:\G_m\to T$ is a cocharacter of $T$, we may define a parabolic subgroup $P_{\chi}$ of $G$ as follows:
$$P_{\chi}=\{g\in G\ /\ \lim_{t\to0}\chi(t)g\chi(t)^{-1} \textrm{ exists}\}.$$
In the above definition, the limit exists if the map $\G_m \to G$, $t\mapsto \chi(t) g 
\chi(t)^{-1}$ extends to $\mathbb{A}^1 \supseteq \G_m$. Note that $P_{\chi}$ contains $T$. Any parabolic subgroup containing $T$ can be defined this way. The set of all possible characters for a given parabolic $P$ is a semigroup with unit a the minimal cocharacter $\chi_P$ such that $P=P_{\chi_P}$. For example, the cocharacter of $P_2^w$ is $w(\chi_{P_2})$. 

\begin{defi}
\label{defi-opp-pairs}
A pair $(P_1,P_2)$ is called \emph{opposite} if $w_0(\chi_{P_1})=-\chi_{P_2}$, where $w_0$ is the longest element of the Weyl group.
\end{defi}

\begin{defi}
We define a parabolic subgroup $R$ of $G$ by its cocharacter 
$$\chi_R=\chi_{P_1} + w(\chi_{P_2}).$$ 
We denote by $L_R$ the Levi subgroup of $R$ containing $T$ and by $U_R$ the unipotent radical of $R$. We have a semidirect product $R=L_R\ltimes U_R$.
\end{defi}

\begin{lemm}
\label{lemm-LK}
Let $w_0^{L_R}$ be the longest element of the Weyl group of $L_R$.

(\i) The parabolic subgroup $R$ contains the intersection $P_1\cap P_2^w$.

(\i\i) The pair $(Q_1,Q_2)$ with $Q_1=L_R\cap P_1$ and $Q_2=(P_2^w\cap L_R)^{w_0^{L_R}}$ is opposite in $L_R$.
\end{lemm}

\begin{proo}
(\i) This is obvious by definition.

(\i\i) We have the equality $\chi_{P_1}\vert_{L_R}+ w(\chi_{P_2})\vert_{L_R}=0$ proving the result.
\end{proo}

\begin{defi}
We set $K = L_R\cap H$. Note that this is the Levi subgroup of both parabolic subgroups of the opposite pair $(Q_1,Q_2)$.
\end{defi}

We have a $G$-equivariant morphism $p:G/H\to G/R$, which is a locally trivial fibration with fiber isomorphic to $R/H$. In other words we have an isomorphism $G/H\simeq G\times^RR/H$. We have a $R$-equivariant morphism $R/H\to L_R/K$ which induces a morphism
$$G/H\simeq G\times^RR/H\to G\times^R L_R/K.$$
Note that since $K=Q_1\cap Q_2$, the diagonal embedding $L_R\to L_R\times L_R$ induces an embedding $L_R/K\to L_R/Q_1\times L_R/Q_2$.

Recall the following definition.

\begin{defi}
\label{def-comin}
A parabolic subgroup is \emph{cominuscule} if its associated cocharacter $\chi_P$ satisfies $\vert\scal{\chi_P,\a}\vert\leq1$ for any root $\a$.
\end{defi}

\begin{lemm}
\label{prop-triv}
(\i) The variety $L_R/K$ is the dense $L_R$-orbit in $L_R/Q_1\times L_R/Q_2$.

(\i\i) The fiber of $R/H\to L_R/K$ is isomorphic to $U_R/U_R\cap H$. It can be embedded in $U_R/U_R\cap P_1\times U_R/U_R\cap P_2^w$.

(\i\i\i) If ${P_1}$ is cominuscule, then the second factor $U_R/U_R\cap P_2^w$ is trivial.
\end{lemm}

\begin{proo}
(\i) Follows from the fact that $Q_1$ and $Q_2$ are opposite.

(\i\i) The statement on the fiber is clear by construction. The claimed embedding is induced by the diagonal embedding $U_R\to U_R\times U_R$.

(\i\i\i) The group $U_R$ is spanned by the groups $U_\a$ for $\a$ a root with $\scal{\chi_R , \a} > 0$ while the group $U_R\cap P_2^w$ is spanned by the groups $U_\a$ for $\a$ a root with $\scal{\chi_R , \a} > 0$ and $\scal{\chi_{P_2^w} , \a} \geq 0$.

Let $\a$ be a root such that $\scal{\chi_R , \a} > 0$ and $\scal{_{P_2^w} , \a} < 0$. Recall that $\chi_R = \chi_{P_1} + \chi_{P_2^w}$ therefore we must have $\scal{\chi_{P_1} , \a} > - \scal{\chi_{P_2^w} , \a} > 0$ and in particular $\scal{\chi_{P_1} , \a} > 1$. A contradiction with the assumption ${P_1}$ cominuscule.
\end{proo}

\begin{coro}
\label{cor-ind}
For $P_1$ and $P_2$ cominuscule, $G/H$ is isomorphic to $G\times^R L_R/K$.
\end{coro}

\begin{rema}
If $P_1$ and $P_2$ are cominuscule, the $G$-orbit $G/H$ is therefore obtained by parabolic induction from $L_R/K$ (see Definition \ref{defi-ind}) that is to say form the case of an opposite pair of parabolic subgroups. 
\end{rema}

\section{Minimal orbits for the weak order}
\label{sec-min}

Let $G$ be a reductive group and $B$ a Borel subgroup. Recall that a $G$-spherical variety, or simply a spherical variety $X$ is a normal $G$-variety with a dense $B$-orbit. This in particular implies that the set $B(X)$ of $B$-orbits is finite. 

In this section we first recall general results on $B$-orbits in a spherical variety $X$. We then apply these results to the case where $X = G/P_1 \times G/P_2$ with $P_1$ and $P_2$ cominuscule parabolic subgroups.

\subsection{Weak order}
\label{subsection-weak}

Let $X$ be a spherical variety and let ${\mathcal{O}}$ be a $B$-orbit in
$X$. There is a natural partial order, called the weak order on the set $B(X)$ of $B$-orbits in $X$ defined as follows. Recall that a minimal parabolic subgroup is a parabolic subgroup with semisimple rank one.

\begin{defi}
\label{defi-weak-order}
Let ${\mathcal{O}}$ be a $B$-orbit in $X$.

(\i) If $P$ is a minimal parabolic subgroup containing $B$ such that
  ${\mathcal{O}}$ is not $P$-stable, we say that $P$ \emph{raises} ${\mathcal{O}}$. 

(\i\i) The \emph{weak order} is the order generated by the following cover relations ${\mathcal{O}}<\mathcal{O'}$ where ${\mathcal{O}}$ is any $B$-orbit in $X$ and where $\mathcal{O'}$ is the dense $B$-orbit in $P{\mathcal{O}}$ for $P$ a minimal parabolic raising ${\mathcal{O}}$. 
\end{defi} 

By results of  \cite{springer?} or \cite{brion1} three cases can occur. Recall that there exists a morphism $P\times^B\mathcal{O}\to P\mathcal{O}$ induced by the action. Recall also that the rank $\rank(Z)$ of a $B$-variety $Z$ is the minimal codimension of $U$-orbits with $U$ the unipotent radical of $B$.

\begin{lemm}
\label{lemm-edges}
Let ${\mathcal{O}}$ be a $B$-orbit in $X$ and let $P$ be a minimal parabolic subgroup raising ${\mathcal{O}}$. Let $\mathcal{O'}$ be the dense $B$-orbit in $P \mathcal{O}$. Then $\dim \mathcal{O}' = \dim \mathcal{O} + 1$ and one of the following three cases occurs:

\emph{(U)} The $P$-orbit $P{\mathcal{O}}$ contains two $B$-orbits ${\mathcal{O}}$ and ${\mathcal{O}}'$ and $P\times^B{\mathcal{O}}\to P\mathcal{O}$ is birational. We have $\rank({\mathcal{O}}')=\rank({\mathcal{O}})$.

\emph{(N)} The $P$-orbit $P{\mathcal{O}}$ contains two $B$-orbits ${\mathcal{O}}$ and ${\mathcal{O}}'$ and $P\times^B{\mathcal{O}}\to P\mathcal{O}$ is of degree~2. We have $\rank({\mathcal{O}}')=\rank({\mathcal{O}})+1$.

\emph{(T)} The $P$-orbit $P{\mathcal{O}}$ contains three $B$-orbits ${\mathcal{O}}$, ${\mathcal{O}}'$ and ${\mathcal{O}}''$ and $P\times^B{\mathcal{O}}\to P\mathcal{O}$ is birational. We have $\dim{\mathcal{O}} = \dim\mathcal{O}''$ and $\rank({\mathcal{O}}')=\rank({\mathcal{O}})+1=\rank({\mathcal{O}}'')+1$.
\end{lemm}

\begin{defi}
We define a graph $\Gamma(X)$ whose vertices are the elements in $B(X)$ and whose edges are the pairs $({\mathcal{O}},{\mathcal{O}}')$ with ${\mathcal{O}}$ raised to ${\mathcal{O}}'$ by a minimal parabolic subgroup $P$. We say that an edge is of type U, N or T if we are in the corresponding U, N or T situation of the previous lemma.
\end{defi}

Let $R$ be a parabolic subgroup of $G$ and let $L_R$ be its Levi quotient. Let $Y$ be a $L_R$-variety. We write $B_{L_R}$ for the image of $B \cap R$ in $L_R$. Note that this is a Borel subgroup of $L_R$.

\begin{defi}
\label{defi-ind}
We say that a $G$-variety $X$ is obtained from $Y$ by \emph{parabolic induction} if of the form $X = G \times^R Y$ where $Y$ is a $L$-variety 
\end{defi}

The following result is a direct application of \cite[Lemma 6]{brion1}. 

\begin{lemm}
\label{lemm-ind}
Let $X$ be a $G$-variety obtained by parabolic induction from $Y$. 

(\i) The variety $X$ is $G$-spherical if and only if $Y$ is $L_R$-spherical.

Assume that $X$ is spherical. 

(\i\i) The set $B(X)$ is in bijection with the product $B_{L_R}(Y)\times B(G/R)$. The bijection $B_{L_R}(Y) \times B(G/R) \to B(X)$ is given by $({\mathcal{O}},B g R/R) \mapsto B g R \times^R {\mathcal{O}}$. Furthermore, the edges are of two types: 
\begin{itemize}
\item either of the form $(({\mathcal{O}},B g R/R),({\mathcal{O}},B g' R/R))$ with $(B g R/R,B g' R/R)$ an edge of $B(G/R)$. These edges are of type \emph{U};
\item or of the form $(({\mathcal{O}},B g R/R),({\mathcal{O}}',B g R/R))$ with $({\mathcal{O}},{\mathcal{O}}')$ an edge of $B_{L_R}(Y)$. The edges  $(({\mathcal{O}},B g R/R),({\mathcal{O}}',B g R/R))$ and $({\mathcal{O}},{\mathcal{O}}')$ have the same type.
\end{itemize}
\end{lemm}

Let $P_1$ and $P_2$ be cominuscule parabolic subgroups and let $X = G/P_1 \times G/P_2$. The following result was proved in \cite{littelmann} (see also \cite{stembridge} for a complete classification of products of projective homogeneous $G$-varieties which are $G$-spherical).

\begin{prop}
The variety $X$ is $G$-spherical.
\end{prop}

Consider a $G$-orbit $G \cdot ( P_1 , w P_2) \simeq G/H$ of $X$ with $w \in W$ and $H=P_1\cap P_2^w$ and recall the notation from Section \ref{section-g-orb}. Corollary \ref{cor-ind} gives the isomorphism 
$$G/H\simeq G\times^RL_R/K.$$ 
In particular, by Lemma \ref{lemm-ind}, to describe the weak order on $G/H$ we only need to study the weak order on $L_R/K$. Thanks to Lemma \ref{lemm-LK}, it is therefore enough to consider the case where $(P_1,P_2)$ is an opposite pair and $w$ is the longest element.

\subsection{Minimal orbits: The case of opposite pairs}

In this subsection, we consider the spherical variety $X = G/P_1 \times G/P_2$ with $P_1$ and $P_2$ two cominuscule parabolic subgroups of $G$ such that $(P_1,P_2)$ is an opposite pair. We pick the dense $G$-orbit in $X$ \emph{i.e.} the orbit $G \cdot ( P_1 , w P_2) \simeq G/H$ with $H=P_1\cap P_2^w$ and $w = w_0$ the longest element of $W$. 

We start with results on minimal length representatives: for $P$ a parabolic subgroup of $G$ containing $B$, we write $W_P$ for its Weyl group and $W^P$ for the subset of $W$ of minimal length representatives of the quotient $W/W_P$.

\begin{lemm}
Let $w_{P_1}$ and $w_{P_2}$ be the longest elements in $W^{P_1}$ and
$W^{P_2}$, then $w_{P_2}=w_{P_1}^{-1}$.
\end{lemm}

\begin{proo}
The length of $w_{P_1}$ and $w_{P_2}$ are equal to the dimensions of
$G/P_1$ and $G/P_2$. Since $(P_1,P_2)$ is an opposite pair, these
dimensions are equal and 
$w_{P_1}(\chi_{P_1})=-\chi_{P_2}$. Thus
$l(w_{P_1}^{-1})=l(w_{P_2})$ and we compute
$w_{P_1}^{-1}(\chi_{P_2})= - \chi_{P_1} = {w_{P_2}}(\chi_{P_2}).$
Therefore $w_{P_1}^{-1}$ is in the same class as $w_{P_2}$ in
$W/W_{P_2}$ proving the result.
\end{proo}

\begin{lemm}
\label{lemm-dual}
 Let $u\in W^{P_1}$, there exists a unique $u^\vee\in W^{P_2}$ such that $(uP_1,u^\vee P_2)$ is in the dense $G$-orbit in $G/P_1\times G/P_2^w$. We have the formulas
$$u^{-1}u^\vee=w_{P_2} \textrm{ and } l(u)+l(u^\vee)=l(w_{P_2}).$$
where $w_{P_2}$ is the longest element in $W^{P_2}$.
\end{lemm}

\begin{proo}
Let $v\in W$ such that $(uP_1,vP_2)$ is in the dense $G$-orbit
\emph{i.e.} have $u(\chi_{P_1})=-v(\chi_{P_2})$. Because $(P_1,P_2)$ is an opposite pair we have $w_{P_1}(\chi_{P_1}) = - \chi_{P_2}$ thus we get $w_{P_1}^{-1}(\chi_{P_2})=u^{-1}v(\chi_{P_2})$ and the equality $w_{P_1}^{-1}=u^{-1}v$ in $W/W_{P_2}$. Let $v'\in W_{P_2}$ such that the equality $w_{P_1}^{-1}=u^{-1}vv'$ holds in $W$. By the previous lemma we get $w_{P_2}=u^{-1}vv'$. 
Write $w_{P_1}=u'u$ with $l(w_{P_1})=l(u)+l(u')$ (this is possible
since $u\in W^{P_1}$). Note that the have $u'={v'}^{-1}v^{-1}$ and
therefore $l(w_{P_2})=l(u^{-1})+l(u')$ and the expression
$w_{P_2}=u^{-1}{u'}^{-1}$ is length additive. Since $w_{P_2}\in
W^{P_2}$ this implies ${u'}^{-1}\in W^{P_2}$. The element
$u^\vee={u'}^{-1}$ satisfies the conclusions of the lemma. 
\end{proo}

\begin{lemm}
\label{lemm-double}
The $B$-orbit $B\cdot(uP_1,{u}^\vee P_2)$ is a $B\times B$-orbit.
\end{lemm}

\begin{proo}
Recall that we have the following equalities
$$B\cdot uP_1=\prod_{\a>0,\ u^{-1}(\a)\not\in P_1}U_{\a}\cdot uP_1
\textrm{ and } B\cdot {u}^\vee
P_2=\prod_{\a>0,\ {{u}^\vee}^{-1}(\a)\not\in P_2}U_{\a}\cdot {u}^\vee
P_2.$$
We are thus left to prove that there is no positive root $\a$
with ${u}^{-1}(\a)\not\in P_1$ and ${u^\vee}^{-1}(\a)\not\in
P_2$. Let $\a$ be such a root. We have the inequalities
$\scal{\chi_{P_1} , u^{-1}(\a)}<0$ and
$\scal{\chi_{P_2} , {u^\vee}^{-1}(\a)}<0$. By Lemma
\ref{lemm-dual}, the second inequality is equivalent to 
$\scal{w_{P_2}(\chi_{P_2}),u^{-1}(\a)}<0$. But since
$w_{P'_2}(\chi_{P_2})=-\chi_{P_1}$ this leads to a
contradiction with the first inequality.
\end{proo}

\begin{lemm}
\label{lemm-closed}
The minimal orbits for the weak order in $G/H$ are
closed. 
\end{lemm}

\begin{proo}
This follows from the fact that this statement holds true for
symmetric homogeneous spaces (see \cite{springer-c}) and the fact that $H$ is a symmetric subgroup: $H$ is the connected component of the subgroup of fixed points under the involution given by conjugation by $\chi_{P_1}(-1)$ (see also \cite[Proposition 3.5]{wahl}).
\end{proo}

\begin{prop}
\label{prop-spec}
The minimal $B$-orbits in $G/H$ are $B\times B$-orbits.
\end{prop}

\begin{proo}
Let $z=(xP_1,yP_2)$ be an element in the dense $G$-orbit of
$G/P_1\times G/P_2$ such that the $B$-orbit $B\cdot z$ is minimal for
the weak order. By letting $B$ act on the first factor, we may assume
that $xP_1$ is fixed by $T$ \emph{i.e.} we have $x\in N_G(T)$. Let $u$ be
its class in the Weyl group $W=N_G(T)/T$. We may assume that $u\in
W^{P_1}$. 

We want to prove that $y$ is also stable by $T$. For this we
introduce the minimal equivariant embedding $G/P_2\subset \p(V_2)$ of
$G/P_2$. The vector space $V_2$ is a representation of $G$ of highest
weight $\varpi_{P_2}$. This is the fundamental weight corresponding to te coweight $\chi_{P_2}$. Let us denote by $\Pi_2$ the set of $T$-weights
of this representation. We have a decomposition 
$$V_2=\bigoplus_{\chi\in\Pi_2}V_2^\chi$$
where $V_2^\chi$ is the eigenspace of weight $\chi$. 
Let $v^i_\chi$ be a basis of the eigenspace $V_2^\chi$ for $\chi\in \Pi_2$. 
We may therefore write $y\cdot v_{\varpi_{P_2}}=\sum_{\chi\in
  \Pi_2}y^i_\chi v_\chi$ with $y^i_\chi$ a scalar. Note that for $\chi$ of the form $W\cdot\varpi_{P_2}$ we have $\dim V_2^\chi=1$ and we write simply $y_\chi$ in that case.

\begin{lemm}
We have $y_{{u}^\vee(\varpi_{P_2})}\neq0$.
\end{lemm}

\begin{proo}
Note that $(uP_1,{u}^\vee P_2)$ and $(xP_1,yP_2)$ are in the dense
$G$-orbit. Since $uP_1=xP_1$ by definition of $u$, we have the
inclusion $yP_2\subset {P_1}^{u} {u}^\vee P_2$. Therefore the class
$[y\cdot v_{\varpi_{P_2}}]$ in $\p(V_2)$ is in the ${P_1}^{u}$-orbit of
the class of a vector of weight ${u}^\vee(\varpi_{P_2})$. 

Consider on the other hand the divisor $D_{{u}^\vee}$ of $\p(V_2)$
defined as the locus of classes $[v]$ of vectors $v$ with trivial
coordinate on $v_{u^\vee(\varpi_{P_2})}$. We claim that this divisor is
${P_1}^{u}$-stable. If this is the case, then $[y\cdot
  v_{\varpi_{P_2}}]$ is not contained in it and the result follows. 

Proving that $D_{{u}^\vee}$ is ${P_1}^{u}$-stable is equivalent to
proving that the weight vector $v_{-{u}^\vee(\varpi_{P_2})}$ of
weight $-{u}^\vee(\varpi_{P_2})$ in the dual space $V_2^\vee$ is
${P_1}^{u}$-stable. The cocharacter defining this stabiliser is
precisely $-{u}^\vee(\chi_{P_2})$ and we have the equalities  
$-{u}^\vee(\chi_{P_2}) = {u}^\vee(w_{P_1}(\chi_{P_1})) =
{u}^\vee(w_{P_2}^{-1}(\chi_{P_1})) = {u}(\chi_{P_1})$
proving the claim.
\end{proo}

As an easy consequence we get that $(uP_1,u^\vee P_2)$ is in the
closure of the $B$-orbit $B\cdot(xP_1,yP_2)$ in $G/H$. Indeed,
choose a one parameter subgroup $\G_m$ of $T$ such that
${u}(\varpi_{P_2})$ has maximal weight on this subgroup. Note that
since $x$ is $T$-stable it is also $\G_m$-stable and that
$[v_{{{u}^\vee(\varpi_{P_2})}}]$ is in the closure of the orbit
$\G_m[y\cdot v_{\varpi_{P_2}}]$. Then letting $\G_m$ act on
$(xP_1,yP_2)$ we get that $(uP_1,u^\vee P_2)$ is in the closure of $B\cdot
(xP_1,yP_2)$ in $X$. Since $(uP_1,u^\vee P_2)$ is in the dense
$G$-orbit $G/H$ it is therefore in the closure of $B\cdot(xP_1,yP_2)$
in $G/H$. 

Since by Lemma \ref{lemm-closed} the orbit $B\cdot(xP_1,yP_2)$ is
closed we get $(uP_1,u^\vee P_2)\in B\cdot(xP_1,yP_2)$. Lemma
\ref{lemm-double} concludes the proof.
\end{proo}

\subsection{Minimal orbits: General case}

In this subsection, we consider the spherical variety $X = G/P_1 \times G/P_2$ with $P_1$ and $P_2$ two cominuscule parabolic subgroups and. We pick a $G$-orbit $G \cdot ( P_1 , w P_2) \simeq G/H$ of $X$ with $H=P_1\cap P_2^w$ and $w \in W$. 

\begin{theo}
\label{theo-spec}
The minimal $B$-orbits in $G/H$ are $B\times B$-orbits.
\end{theo}

\begin{proo}
According to Lemma \ref{lemm-ind}, a minimal $B$-orbit is of the form $B g R \times^R {\mathcal{O}}$ where $B g R/R$ is a minimal $B$-orbit in $B(G/R)$ and $\mathcal{O}$ is a minimal $B_{L_R}$-orbit in $L_R/K$. Therefore $B g R /R$ is a point and $\mathcal{O}$ is a $B_{L_R} \times B_{L_R}$-orbit. The result follows.
\end{proo}

\section{Distance and rank}
\label{sec-dist}

In this section we consider $X = G/P_1 \times G/P_2$ with $G$ simply laced and $P_1$, $P_2$ cominuscule. We prove that there is no edge of type N in the graph
$B(X)$. 

By definition of the weak order, we only need to consider $B(G/H)$ for $G/H$ a $G$-orbit with $H = P_1 \cap P_2^w$ in $X$. Note that thanks to Lemma \ref{lemm-ind} and Corollary \ref{cor-ind}, we only need to prove this result for opposite pairs. In all the section we assume that $P_1$ and $P_2$ are cominuscule and shall specify when we assume that the pair $(P_1,P_2)$ is an opposite pair.

\subsection{Distance}

In this subsection we introduce a \emph{distance} $d(x,y)$ between
$T$-fixed points $xP_1\in G/P_1$ and $yP_2\in G/P_2$ and prove that it
is closely related to the rank of the $B$-orbit of $(xP_1,yP_2)$. Let $\varpi_{P_i}$ be the fundamental weight corresponding to the cocharacter $\chi_{P_i}$. Denote by
$V_{\varpi_{P_i}}$ the irreducible representation of highest weight
${\varpi_{P_i}}$ and by $\Pi_{\varpi_{P_i}}$ the set of weights of
$V_{\varpi_{P_i}}$ for $i\in\{1,2\}$. Recall that $W\cdot\varpi_{P_i}$
the $W$-orbit of $\varpi_{P_i}$ is equal to $\Pi_{\varpi_{P_i}}$ in
our situation since $G$ is simply laced and both weights are
cominuscule therefore minuscule. Recall also that the map $W^{P_i} \to \Pi_{\varpi_i}, u \mapsto u(\varpi_{P_i})$ is bijective and that the Schubert cells in $G/P_i$ are of the form $\Omega_u = B u P_i/P_i$ for a unique $u \in W^{P_i}$. Fix $(\ ,\ )$ a $W$-invariant scalar product and write $\vert\cdot \vert$ for the associated norm. 

\begin{defi}
For $\lambda_i\in\Pi_{\varpi_{P_i}}$ define
$d(\lambda_1,\lambda_2)=(\varpi_{P_1},\varpi_{P_2})- (\lambda_1,\lambda_2)$.
\end{defi}

\begin{rema}
(\i) The distance $d(\lambda_1,\lambda_2)$ is $W$-invariant.

(\i\i) If $\varpi_{P_1}=\varpi_{P_2}$, then we have 
  $d(\lambda_1,\lambda_2)=\frac{1}{2}\vert\lambda_1-\lambda_2\vert^2$.
\end{rema}

\begin{lemm}
We have
$d(\lambda_1,\lambda_2)\in[0,(\varpi_{P_1},\varpi_{P_2}-w_{P_2}(\varpi_{P_2}))]$. 
\end{lemm}

\begin{proo}
Since the distance is $W$-invariant, we have $d(\lambda_1,\lambda_2)=d(\varpi_{P_1},\mu)$ for some $\mu\in\Pi_{\varpi_{P_2}}$. We have $d(\varpi_{P_1},\mu)=(\varpi_{P_1},\varpi_{P_2}-\mu)$. Since $\varpi_{P_2}$ is the highest weight of $V_{\varpi_{P_2}}$ and $w_{P_2}(\varpi_{P_2})$ the lowest weight, the result follows. 
\end{proo}

\begin{lemm}
We have $d(\lambda_1,\lambda_2)=0$ if and only if $\lambda_1$ and $\lambda_2$ belong to the same chamber.
\end{lemm}

\begin{proo}
If $\lambda_1$ and $\lambda_2$ belong to the same chamber, then letting $W$ act we may assume that this chamber is the dominant chamber. In particular $\lambda_i=\varpi_{P_i}$ and the distance vanishes. Conversely, we may assume by letting $W$ act that $\lambda_1=\varpi_{P_1}$. We proceed by induction on $\varpi_{P_2}-\lambda_2$. If $\lambda_2=\varpi_{P_2}$, we are done. Otherwise $\lambda_2<\varpi_{P_2}$ and there exists a simple root $\a$ such that 
$$\lambda_2< s_\a(\lambda_2)=\lambda_2+\a\leq\varpi_{P_2}.$$
Furthermore, since $d(\lambda_1,\lambda_2)=d(\varpi_{P_1},\lambda_2)=(\varpi_{P_1},\varpi_{P_2}-\lambda_2)=0$ we must have $(\varpi_{P_1},\a)=0$. Then we have $0=d(s_\a(\varpi_{P_1}),s_\a(\lambda_2))=d(\varpi_{P_1},s_\a(\lambda_2))$. By induction, $\varpi_{P_1}$ and $s_\a(\lambda_2)$ are in the same chamber. The same is therefore true for $s_\a(\varpi_{P_1})=\varpi_{P_1}$ and $\lambda_2$.
\end{proo}

\begin{coro}
If $d(\lambda_1,\lambda_2)>0$, then there exists a root $\a$ with $(\lambda_1,\alpha)(\lambda_2,\a)<0$.
\end{coro}

\begin{proo}
If there is no root $\a$ with $(\lambda_1,\alpha)(\lambda_2,\a)<0$, then $\lambda_1$ and $\lambda_2$ are in the same chamber and $d(\lambda_1,\lambda_2)=0$ by the previous lemma.
\end{proo}

\begin{lemm}
\label{lemme-facile}
For $(\lambda_1,\alpha)(\lambda_2,\a)<0$, we have $d(\lambda_1,s_\a(\lambda_2))=d(\lambda_1,\lambda_2)-1$.
\end{lemm}

\begin{proo}
For $P_i$ cominuscule and $G$ simply laced, we have $(\lambda_i,\alpha)\in\{-1,0,1\}$. The result follows from this by an easy computation. 
\end{proo}

\begin{coro}
Let $d=d(\lambda_1,\lambda_2)$.

(\i) There exists a sequence $(\gamma_i)_{i\in[1,d]}$ or roots such that if $(\mu_i)_{i\in[0,d]}$ is defined by $\mu_d=\lambda_2$ and $\mu_{i-1}=s_{\gamma_i}(\mu_i)$, then $d(\lambda_1,\mu_i)=i$.

(\i\i) The roots $(\gamma_i)_{i\in[1,d]}$ are mutually orthogonal and satisfy $(\lambda_1,\gamma_i)(\lambda_2,\gamma_i)<0$ for all $i\in[1,d]$.
\end{coro}

\begin{proo}
(\i) We proceed by induction on $d$. By the former corollary, if $d>0$, there exists a root $\a$ with $(\lambda_1,\a)(\lambda_2,\a)<0$. Set $\gamma_d=\a$ and $\mu_{d-1}=s_\a(\lambda_2)$, then $d(\lambda_1,\mu_{d-1})=d-1$. We conclude by induction.

(\i\i) Note that in the sequence $(\gamma_k)_{k\in[1,d]}$, we may replace $\gamma_k$ by its opposite. Therefore we may assume that $(\lambda_1,\gamma_k)<0$ (and thus $(\mu_k,\gamma_k)>0$) for all $i\in[1,d]$.
We first prove by induction on $j-i$ the vanishing $(\gamma_i,\gamma_j)=0$ for all $i<j$. 
By induction assumption, we have
$$\mu_{i}=s_{\gamma_{i+1}}\cdots s_{\gamma_j}(\mu_j)=\mu_j-\sum_{k=i+1}^{j}(\gamma_k,\mu_k)\gamma_k=\mu_j-\sum_{k=i+1}^{j}\gamma_k.$$
We get, again using induction
$$1\geq(\gamma_i,\mu_j)=(\gamma_i,\mu_i)+\sum_{k=i+1}^{j}(\gamma_k,\mu_k)(\gamma_i,\gamma_k)=1+(\gamma_i,\gamma_j).$$
In particular we get $(\gamma_i,\gamma_j)\leq0$. If $(\gamma_i,\gamma_j)=-1$, then $\gamma_i+\gamma_j$ would be a root and we would have $(\lambda_1,\gamma_i+\gamma_j)\geq-1$. But $(\lambda_1,\gamma_i+\gamma_j)=-2$ a contradiction. The second condition easily follows.
\end{proo}

We can prove a converse of the above statement.

\begin{lemm}
If $(\gamma_i)_{i\in[1,d]}$ is a sequence  of 
mutually orthogonal roots such that for all $i\in[1,d]$, we have $(\lambda_1,\gamma_i)(\lambda_2,\gamma_i)<0$, then $d(\lambda_1,\lambda_2)\geq d$.
\end{lemm}

\begin{proo}
Define the sequence $(\mu_i)_{i\in[0,d]}$ of weights as above: $\mu_d=\lambda_2$ and $\mu_{i-1}=s_{\gamma_i}(\mu_i)$. We have $d(\lambda_1,\mu_{i+1})=d(\lambda_1,\mu_i)-1$ for all $i$, the result follows.
\end{proo}

\begin{coro}
\label{coro-dist}
The distance $d(\lambda_1,\lambda_2)$ is the maximal length of sequences $(\gamma_i)_{i\in[1,d]}$ of mutually orthogonal roots satisfying $(\lambda_1,\gamma_i)(\lambda_2,\gamma_i)<0$ for all $i\in[1,d]$.
\end{coro}

\subsection{Connection with the rank}

Let $B(X)$ be the set of $B$-orbits in $X=G/P_1\times G/P_2$. We
define a map $\Phi:B(X)\to W^{P_1}\times W^{P_2}$ as follows. Let $\mathcal{O}$ be a $B$-orbit in $G/P_1\times G/P_2$. Then the images of ${\mathcal{O}}$ in $G/P_1$ and in $G/P_2$ are Schubert cells $\Omega_u$ and $\Omega_v$ with $(u,v)\in W^{P_1}\times W^{P_2}$. We put 
$$\Phi({\mathcal{O}})=(u,v).$$

\begin{rema}
We defined the distance on the pairs of weights in $\Pi_1\times\Pi_2$.
We extend this definition to $W^{P_1}\times W^{P_2}$ by setting
$d(u,v)=d(u(\varpi_{P_1}),v(\varpi_{P_2})).$
\end{rema}

\begin{lemm}
\label{lemm-ineq}
Let ${\mathcal{O}},\mathcal{O'}\in B(X)$ with ${\mathcal{O}}\leq{\mathcal{O}}'$ for the weak order. Let $(u,v)=\Phi(\mathcal{O})$ and $(u',v')=\Phi(\mathcal{O}')$. We have $d(u,v)-d(u',v')\leq \rank({\mathcal{O}}')-\rank(\mathcal{O})$.
\end{lemm}

\begin{proo}
Choose a sequence $(P_{\gamma_i})_{i\in[1,r]}$ of minimal parabolic subgroups raising ${\mathcal{O}}$ to ${\mathcal{O}}'$. Here $\gamma_i$ for $i \in [1,r]$ denotes the simple root whose opposite is a root of $P_{\gamma_i}$. Let us write ${\mathcal{O}}_i$ for the dense $B$-orbit in $P_{\gamma_i}\cdots P_{\gamma_1}{\mathcal{O}}$ and write $\Phi({\mathcal{O}}_i)=(u_i,v_i)$. We have the three possibilities: 
\begin{itemize}
\item if $(\gamma_{i+1},u_i(\varpi_{P_1}))=1$, then we have $u_{i+1} = s_{\gamma_{i+1}}u_i$ and $u_{i+1}(\varpi_{P_1}) = s_{\gamma_{i+1}}u_i(\varpi_{P_1}) = u_i(\varpi_{P_1})-\gamma_{i+1}$,
\item if $(\gamma_{i+1},u_i(\varpi_{P_1})) = 0$, then we have $u_{i+1} = u_i$ and $u_{i+1}(\varpi_{P_1}) = u_i(\varpi_{P_1}) = s_{\gamma_{i+1}}u_i(\varpi_{P_1})$,
\item if $(\gamma_{i+1},u_i(\varpi_{P_1}))=-1$, then we have $u_{i+1}=u_i$ and $u_{i+1}(\varpi_{P_1})=u_i(\varpi_{P_1})$.
\end{itemize}
The same possibilities occur for $v_i$. There are only two cases for which we have $d(u_{i+1},v_{i+1})\neq d(u_i,v_i)$, namely for $(\gamma_{i+1},u_i(\varpi_{P_1}))=1$ and $(\gamma_{i+1},v_i(\varpi_{P_2}))=-1$ and for $(\gamma_{i+1},u_i(\varpi_{P_1}))=-1$ and $(\gamma_{i+1},v_i(\varpi_{P_2}))=1$. In both cases we have $d(u_{i+1},v_{i+1})=d(u_{i},v_{i})-1$ by Lemma \ref{lemme-facile}. 

We claim that the following inequality holds
$$\rank({\mathcal{O}}_{i+1})- \rank({\mathcal{O}}_{i})\geq d(u_i,v_i)-d(u_{i+1},v_{i+1}).$$
Since $\rank({\mathcal{O}}_{i+1})\geq \rank({\mathcal{O}}_{i})$ this is clear in all cases where $d(u_{i+1},v_{i+1})=d(u_{i},v_{i})$. The last two cases are symmetric, we only treat the first one. Remark that the orbit ${\mathcal{O}}_{i+1}=P_{\gamma_{i+1}}{\mathcal{O}}_i$ contains the orbit ${\mathcal{O}}_i$ and another orbit. Indeed, if $y$ is the $T$-fixed element in $\Omega_{v_i}$, then there exists an element of the form $(x,y)$ in ${\mathcal{O}}_i$. The element $s_{\gamma_{i+1}}(x,y)$ is in ${\mathcal{O}}_{i+1}$ and $s_{\gamma_{i+1}}(y)$ is a $T$-fixed point different from $y$. Since the image by the second projection to $G/P_2$ of ${\mathcal{O}}_i$ and ${\mathcal{O}}_{i+1}$ is $\Omega_{v_i}$ which does not contain $s_{\gamma_{i+1}}(y)$ there is a third orbit ${\mathcal{O}}'_i$ contained in ${\mathcal{O}}_{i+1}$ and containing $s_{\gamma_{i+1}}(y)$. In particular $\rank({\mathcal{O}}_{i+1})=\rank({\mathcal{O}}_i)+1$. The claim is proved.

Summing up we get the desired inequality.
\end{proo}

\begin{prop}
\label{prop-ineq}
Assume that $P_1$ and $P_2$ are opposite and $w$ is the longest
element. We have the inequality
$d(\varpi_{P_1},w(\varpi_{P_2}))\geq\rank(X)$.
\end{prop}

\begin{proo}
Consider the dense $G$-orbit $G/H$ with $H = P_1 \cap P_2^{w_0}$ in $X$. This is the orbit of $([w_{P_1}(\varpi_{P_1})],[\varpi_{P_2}])$. We have a surjective morphism $p_1:G/H\to G/P_1$ and we consider the fiber of $[w_{P_1}(\varpi_{P_1})]$ which is isomorphic to $P_1^{w_{P_1}}\cdot[\varpi_{P_2}]\simeq P_1^{w_{P_1}}/P_1^{w_{P_1}}\cap P_2\simeq P_2^-/P_2^-\cap P_2\simeq L_2 U_{P_2}/L_2$ where $U_{P_2}$ is the unipotent radical of $P_2$ and $L_2$ is the Levi subgroup containing $T$. We have a trivialisation of the morphism $p_1:G/H\to G/P_1$ over the open subset $U_{P_1}\cdot[w_{P_1}(\varpi_{P_1})]\simeq U_{P_1}$ and therefore an open $B$-stable subset of $X$ isomorphic to
$$U_{P_1}\times L_2U_{P_2}/L_2.$$
The rank of $X$ as a $G$-variety is therefore the rank of $L_2U_{P_2}/L_2\simeq U_{P_2}$ as an $L_2$-variety. The action on $U_{P_2}$ is by conjugation. To compute the rank we want to compute the dimension of the quotient $U_{P_2}/U$ where $U$ is the opposite maximal unipotent of $L_2$. 

Let us note that $U_{P_2}$ is a vector space direct sum of the $U_\a$ for $(\a,\varpi_{P_2})=1$. The action of $U_\beta\subset U$ on $U_{P_2}$ induces a morphism $U_\beta\times U_\a\to U_\a\times U_{\a+\beta}$ defined by $(b,a)\mapsto (a,c_{\a,\beta}ab)$ for some constant $c_{\a,\b}$ (non vanishing if $\a+\b$ is a root, see \cite[Proposition 8.2.3]{springer}). We construct subspaces of $U_{P_2}$ stable for the action of $U$.

We define a sequence $(R_i,\theta_i)_{i\in[1,r]}$ of pairs consisting of a root system $R_i$ and a root $\theta_i\in R_i$ by induction. Let $R_1=R$ be the root system of $G$ and let $\theta_1$ be the highest root of $R_1$. Define $R_{i+1}$ as the root system of all roots orthogonal to $\theta_i$ and $\theta_{i+1}$ be the highest root in $R_{i+1}$.

\begin{lemm}
\label{lemm-1}
Let $\a=\theta_i-\beta$ with $\a$ and $\b$ two roots. Then $(\theta_i,\a)=(\theta_i,\b)=1$. Conversely, for $\a$ a root, if $(\a,\theta_i)=1$, then $\b=\theta-\a$ is a root.
\end{lemm}

\begin{proo}
Since $\b$ is a root we have $2=(\b,\b)=4-2(\theta_i,\a)$ proving the first equality. A similar argument gives the second proof. For the converse write $\b=\theta_i-\a=s_\a(\theta_i)$.
\end{proo}

\begin{lemm}
\label{lemm-2}
The following conditions are equivalent
\begin{itemize}
\item $\a\in R$ and $\a\leq\theta_i$;
\item $\a\in R_i$.
\end{itemize}
\end{lemm}

\begin{proo}
The second condition implies obviously the first by definition of $\theta_i$. Conversely, note that for the root system $R_k$, the root $\theta_k$ is in the dominant chamber thus if $\gamma_1,\cdots,\gamma_r$ are the simple roots of $R_k$ orthogonal to $\theta_k$, then the roots of $R_{k+1}$ are exactly the roots with trivial coefficient on the simple roots $\gamma_1,\cdots,\gamma_r$. This in particular implies that if $\a\leq\theta_i$, then $\a\in R_i$.
\end{proo}

\begin{lemm}
\label{lemm-3}
We have $R_i\setminus R_{i+1}=\{\a\in R\ /\ \exists \textrm{ $\b$ positive root with }\a=\theta_i-\b\}.$
\end{lemm}

\begin{proo}
  If $\a=\theta_i-\b$, then $\a\leq\theta_i$ and by Lemma
  \ref{lemm-2} $\a\in R_i$. Furthermore by Lemma \ref{lemm-1} we have
  $(\a,\theta_i)=1$ thus $\a\not\in R_{i+1}$ Conversely, if $\a\in
  R_i\setminus R_{i+1}$, we have $(\a,\theta_i)\neq0$ thus since
  $\theta_i$ is the highest root $(\a,\theta_i)=1$. By Lemma
  \ref{lemm-1} there is a root $\b$ with $\a=\theta_i-\b$ and $\b\in
  R_i$. Since $\theta_i$ is the highest root of $R_i$ we have $\b>0$.
\end{proo}

\begin{lemm}
\label{lemm-4}
Let $\a_i$ and $\a_j$ two roots of $U_{P_2}$ and assume that
  $\a_i=\theta_i-\b_i$ and $\a_j=\theta_j-\b_j$ for $\b_i$, $\b_j$
  positive roots.

(\i) If $i\neq j$. Then we have $\b_i\neq\b_j$.

(\i\i) If $i= j$. Then we have $(\b_i,\b_j)=0$.
\end{lemm}

\begin{proo}
(\i) We may assume $i<j$. Assume further that $\b_i=\b_j=\b$ and recall
that we have $(\a_i,\theta_i)=1=(\a_j,\theta_j)$,
  $(\a_j,\theta_i)=(\theta_i,\theta_j)=0$. We may compute
$$(\b,\b)=(\theta_i-\a_i,\theta_j-\a_j)=(\a_i,\a_j)-(\a_i,\theta_j).$$
But since $\a_j$ is in $U_{P_2}$, the same is true for $\theta_j$ and
therefore $(\a_i,\theta_j)=0$. We get $2=(\a_i,\a_j)$ which would imply
$\a_i=\a_j$ a contradiction.

(\i\i) We have
$(\b_i,\b_j)
=(\theta_i,\theta_i)-(\theta_i,\a_j)-(\a_i,\theta_i)+(\a_i,\a_j)=2-1-1+0$.  
\end{proo}

We set $U(\theta_i)=\prod_{\a\in R_i\setminus R_{i+1},\ U_\a\subset
  U_{P_2}}U_\a\subset U_{P_2}$ for $i\in[1,s]$ with $s\leq r$ such
that $U(\theta_i)$ is not trivial for $i\leq s$. These are subspaces of
$U_{P_2}$. Note that for $\a,\b$ roots of $U_{P_2}$ we have $U_\a
U_\b=U_\b U_\a$ so that we do not have to take care of the order of
the product. We also define for $i\in[1,s]$
$$U_i=\prod_{\b\in R^+,\ \theta_i-\b\in R_i}U_{-\b}.$$
Note that by Lemma \ref{lemm-4}, the above $U_{-\b}$ commute so that
 we can take any order for this product.

\begin{lemm}
   In the $U_i$-orbit of a general element in $U_{P_2}$ there is a
   unique representative whose only non trivial coordinate in
   $U(\theta_i)$ lies in $U_{\theta_i}$.
\end{lemm}

\begin{proo}
  Indeed, choose an element with non trivial coordinate in
  $U_{\theta_i}$. Letting $U_i$ act we can kill all the other
  coordinates in $U(\theta_i)$ in a unique way.
\end{proo}
 
\begin{lemm}
   In the $U$-orbit of a general element in $U_{P_2}$ there is a
   representative whose only non trivial coordinates are in $\prod_{i=1}^s
   U_{\theta_i}$.
\end{lemm}

\begin{proo}
  Apply the previous Lemma by induction.
\end{proo}

In particular, we see that $\rank(X)=\dim U_{P_2}/U\leq s$. But
$(\theta_i)_{i\in[1,s]}$ is a sequence of mutually orthogonal roots
with $(\theta_i,\varpi_{P_2})=1$ and
$(\theta_i,w_{P_1}(\varpi_{P_1}))=(\theta_i,-\varpi_{P_2})=-1$ thus by
Corollary \ref{coro-dist} we have
$d(\varpi_{P_1},w_{P_2}(\varpi_{P_2}))=d(w_{P_1}(\varpi_{P_1}),\varpi_{P_2})\geq
s$ and the proposition is proved.
\end{proo}

\begin{theo}
Assume that $P_1$ and $P_2$ are opposite and $w$ is the longest
element. 
  Let ${\mathcal{O}}\in B(X)$ and set $\Phi({\mathcal{O}})=(u,v)$. Then
  $\rank({\mathcal{O}})+d(u,v)=\rank(X)$. 
\end{theo}

\begin{proo}
  As in the proof of Lemma \ref{lemm-ineq}, we may choose a sequence
  $(P{\gamma_i})_{i\in[1,r]}$ of minimal parabolic subgroups raising a minimal orbit
  ${\mathcal{O}}'$ for the weak order to $X$ such that if we write
  ${\mathcal{O}}_i$ for the dense $B$-orbit in $P_{\gamma_i}\cdots
  P_{\gamma_1}{\mathcal{O}}'$ there is an index $k$ such that ${\mathcal{O}}_k={\mathcal{O}}$. Set $\Phi({\mathcal{O}}_i)=(u_i,v_i)$. According to the
  proof of Lemma \ref{lemm-ineq}, the equality
  $d(u_{i+1},v_{i+1})=d(u_i,v_i)-1$ implies the equality $\rank({\mathcal{O}}_{i+1})=\rank({\mathcal{O}}_i)+1$. In particular, we get
  $d(u_0,v_0)=d(u_0,v_0)-d(u_r,v_r)\leq \rank(X)-\rank({\mathcal{O}}')\leq\rank(X)\leq d(1,w_{P_2})$. But since ${\mathcal{O}}'$ is
  minimal for the weak order we have by Theorem \ref{theo-spec} the
  equality $v_0=u_0^\vee$ and by Lemma \ref{lemm-dual} we have
  $u_0^{-1}u_0^\vee=w_{P_2}$. Therefore
  $d(u_0,v_0)=d(1,u_0^{-1}v_0)=d(1,w_{P_2})$ and we have equality in
  all the inequalities. The result follows.
\end{proo}

\begin{coro}
\label{N-coro}
  There is no edge of type $N$ in the graph $\Gamma(X)$.
\end{coro}

\begin{proo}
By Lemma \ref{lemm-ind} and Corollary \ref{cor-ind}, we may assume that $P_1$ and $P_2$ are
opposite and $w$ is the longest element.

  Choose any minimal orbit ${\mathcal{O}}$ in $X$ and any sequence
  $(P_{\gamma_i})_{i\in[1,r]}$ of minimal parabolic subgroups raising ${\mathcal{O}}$ to
  $X$. Write ${\mathcal{O}}_i$ for the dense $B$-orbit in
  $P_{\gamma_i}\cdots P_{\gamma_1}{\mathcal{O}}$ and set $\Phi({\mathcal{O}}_i)=(u_i,v_i)$. According to the proof of Lemma \ref{lemm-ineq}
  and to Lemma \ref{lemme-facile}, the equality
  $d(u_{i+1},v_{i+1})=d(u_i,v_i)-1$ implies the equality $\rank({\mathcal{O}}_{i+1})=\rank({\mathcal{O}}_i)+1$ and occurs only when
  $(u_{i}(\varpi_{P_1}),\gamma_{i+1})(v_{i}(\varpi_{P_2}),\gamma_{i+1})<0$. All the edges corresponding to such a raising by $P_{\gamma_{i+1}}$ are of type T by the above proof. But since $d(u_0,v_0)=\rank(X)-\rank({\mathcal{O}})$ there is no other edge of $\Gamma(X)$ raising the rank. Since edges of type N raise the rank there is no such edge.
\end{proo}

Let $(P_{\gamma_i})_{i\in[1,r]}$ be a sequence of parabolics raising
the orbit ${\mathcal{O}}$ to $P_{\gamma_r}\cdots P_{\gamma_1}{\mathcal{O}}$.

\begin{coro}
  The map
  $\pi:P_{\gamma_r}\times^B\cdots\times^BP_{\gamma_1}\times^B {\mathcal{O}}\to P_{\gamma_r}\cdots P_{\gamma_1}{\mathcal{O}}$ is
  birational.
\end{coro}

Let $Y$ be the closure of a $B$-orbit in $X$. 

\begin{coro}
  The normalisation morphism $\nu:\widetilde{Y}\to Y$ is an
  homeomorphism.
\end{coro}

\begin{proo}
Let ${\mathcal{O}}$ be the dense $B$-orbit in $Y$. There exists a minimal $B$-orbit ${\mathcal{O}}'$ that can be raised to ${\mathcal{O}}$. Let $(P_{\gamma_i})_{i\in[1,r]}$ be a sequence of parabolics raising the orbit ${\mathcal{O}'}$ to ${\mathcal{O}}$.
The closure $Y'$ of ${\mathcal{O}}'$ is a product of Schubert varieties therefore normal and there is a morphism birational with connected fibers
$P_{\gamma_r}\times^B\cdots\times^B P_{\gamma_1}\times^BY'\to Y$. The result follows. 
\end{proo}

\section{Proof of Theorem \ref{main-intro}}
\label{sec-fin}

We want to use the technique developed in \cite{brion1} and \cite{brion2} to prove normality of the $B$-orbit closures. In particular Brion proves the following.

\begin{prop}
\label{prop-mb}
Let $X$ be $G$-spherical variety such that the graph $\Gamma(X)$ has no edge of type $N$. Let $Y$ be a $B$-stable subvariety such that for all minimal parabolic subgroups $P$ raising $Y$ the variety $PY$ is normal, then the non normal locus in $Y$ is $G$-invariant.
\end{prop}

We will use the following consequence of this result.

\begin{coro}
\label{coro-sans-z}
Assume that $X$ is $G$-spherical with a unique closed $G$-orbit $Z$ and such that the graph $\Gamma(X)$ has no edge of type $N$. If any $B$-orbit closure containing $Z$ is normal, then any $B$-orbit closure is normal.
\end{coro}

Consider $X = G/P_1 \times G/P_2$ with $P_1$ and $P_2$ cominuscule. The variety $X$ is $G$-spherical and has a unique closed $G$-orbit $Z$ obtained as the image of the map $G/P_1 \cap P_2$ induced by the diagonal embedding $G \to G \times G$. To prove Theorem \ref{main-intro}, we therefore only have to prove the normality of $B$-orbit closures containing $Z$.

Let $Y'$ be a $B$-orbit closure containing $Z$. There exists a minimal orbit closure $Y$ and a sequence of minimal parabolic subgroups $(P_{\gamma_i})_{i\in[1,r]}$ such that with $Y_0=Y$ and $Y_i=P_{\gamma_i}\cdots P_{\gamma_1} Y$ for $i\geq 1$, the parabolic $P_{\gamma_{i+1}}$ raises $Y_i$ to $Y_{i+1}$ for all $i$ and $Y_r=Y'$.

\begin{prop}
The inverse image $\pi^{-1}(Z)$ of $Z$ by $\pi:P_{\gamma_r} \times^B \cdots \times^B P_{\gamma_1} \times^B Y \to Y'$ is 
reduced. 
Furthermore, the differential of the map $P_{\gamma_r}\times^B\cdots\times^BP_{\gamma_1}\times^B (Z\cap Y)\to Z$ is generically surjective.
\end{prop}

\begin{proo}
Since $Z$ is $G$-stable, the inverse image of $Z$ by the action $G \times X \to X$ is $G \times Z$. This implies that the inverse image $\pi^{-1}(Z)$ has to be contained in $P_{\gamma_r}\times^B\cdots\times^BP_{\gamma_1}\times^B (Z\cap Y)$ and thus isomorphic to it. But $Y$ is a minimal $B$-orbit closure and as such (Theorem \ref{theo-spec}) is a product $X_u^{P_1}\times X^{P_2}_v$ of Schubert varieties (we write here $X_u^P$ for the orbit closure of $BuP/P$ in $G/P$). The intersection with the closed orbit is therefore an intersection $\Upsilon$ of two Schubert varieties for $B$ in $Z = G/P_1 \cap P_2$ and in particular reduced. The above also implies that the map $P_{\gamma_r}\times^B\cdots\times^BP_{\gamma_1}\times^B (Z\cap Y)\to Z$ is the classical multiplication map $P_{\gamma_r}\times^B\cdots\times^BP_{\gamma_1}\times^B \Upsilon \to G/P_1\cap P_2$ obtained by a partial Bott-Samelson resolution which has a generically surjective differential.
\end{proo}

\begin{coro}
\label{coro-rec}
Let $Y'$ be a $B$-orbit closure in $X$ containing $Z$ such that for any parabolic subgroup $P$ raising $Y'$, the variety $PY'$ is normal, then $Y'$ is normal.
\end{coro}

\begin{proo}
Indeed, by Proposition \ref{prop-mb}, the non normal locus of $Y'$ is
$G$-invariant and therefore contains $Z$. Let $\nu:\widetilde{Y}'\to
Y'$ be the normalisation. The map $\nu$ is bijective. We therefore
only have to prove that $\nu$ is an isomorphism on an open subset of
$Z$. In particular, we only have to prove that the general fiber of
$\nu$ over $Z$ is reduced and that the differential of $\nu$ is
generically surjective on $Z$. But there exists $Y$, the closure of a
minimal $B$-orbit and a sequence $(P_{\gamma_i})_{i\in[1,r]}$ of minimal
parabolic subgroups as in the previous Proposition. Furthermore, the
morphism $\pi:P_{\gamma_r}\times^B\cdots\times^BP_{\gamma_1}\times^B Y\to Y'$ factorises
through $\nu$. This finishes the proof. 
\end{proo}

\begin{proo-thm}
We prove the normality of $B$-orbit closures by descending induction with respect to the weak order. 

A maximal $B$-orbit $\mathcal{O}$ is a $G$-orbit therefore of the form $G/H$ with $H = P_1 \cap P_2^w$. We thus have $\mathcal{O} \simeq G \times^{P_1} P_1 P_2^w/P_2^w$. The closure is then a locally trivial fibration over $G/P_1$ with fiber the Schubert variety $\overline{P_1P_2^w}/P_2^w$. It is normal since Schubert varieties are normal by \cite{MR}. 

Let $Y$ be a $B$-orbit closure, then by Corollary \ref{coro-sans-z}, the non normal locus must be closed and $G$-stable. It is therefore either empty or contains the unique closed orbit $Z$. If $Y$ does not contain $Z$, then it must be normal. If $Y$ contains $Z$, then by induction assumption, the hypothesis of Corollary \ref{coro-rec} are satisfied and $Y$ is normal. 

The Cohen-Macaulay property follows from general argument in \cite[Section 3, Remark 2]{brion2}. It will also follow from the existence of a rational resolution. For this, let $Y'$ be a $B$-orbit closure and $Y$ and $(P_{\gamma_i})_{i\in[1,r]}$ be the closure of a minimal $B$-orbit and a sequence of minimal parabolics raising $Y$ to $Y'$. The variety $Y$ is a product of Schubert varieties by Theorem \ref{theo-spec}. Let $\widetilde{Y}$ be the product of the Bott-Samelson resolutions of these varieties. Then by the same arguments as in \cite[Section 3, end of Remark 2]{brion2} the morphism $P_{\gamma_r}\times^B\cdots \times^B P_{\gamma_1} \times^B\widetilde{Y}\to Y'$ is a rational resolution.
\end{proo-thm}

\begin{rema}
It would be interesting to obtain a proof of Theorem \ref{main-intro} in the spirit of \cite{MR}, using Frobenius splitting techniques. However, we were not able to find a Frobenius splitting of $X = G/P_1 \times G/P_2$ with $P_1$ and $P_2$ cominuscule which compatibly splits the $B$-orbit closures. Note however that in type A there exists a Frobenius splitting which compatibly splits the $B$-orbit closures containing the closed orbit $Z$. This can be used to give an alternative proof of Theorem \ref{main-intro} in type~A.
\end{rema}

\section{Example of non normal closures}
\label{sec-exa}

In this section we give an counterexample to Theorem \ref{main-intro} and Corollary \ref{N-coro} for $G$ non simply laced.

Let $(e_i)_{i\in[1,6]}$ be the canonical basis in $k^6$. Define the symplectic form $\omega$ on $k^6$ by $\omega(e_i,e_j)=\delta_{7,i+j}$ for all $i < j$. Let $G$ be the symplectic group ${\rm Sp}_6$ of linear automorphisms preserving $\omega$. Let $P_1=P_2$ be the stabiliser of the $3$-dimensional isotropic subspace $\scal{e_1,e_2,e_3}$. Then $X=G/P_1\times G/P_2$ is the set of pairs of maximal (of dimension $3$) subspaces in $k^6$ isotropic for $\omega$. Consider the full isotropic flag
$$\scal{e_1}\subset \scal{e_1,e_2}\subset \scal{e_1,e_2,e_3}\subset \scal{e_1,e_2,e_3,e_4}\subset \scal{e_1,e_2,e_3,e_4,e_5}\subset k^6$$
and the Borel subgroup $B$ of $G$ stabilising this complete flag. We denote by $T$ the maximal torus defined by the basis $(e_i)_{i\in[1,6]}$. We denote by $\a_1$, $\a_2$ and $\a_3$ the simple roots of $G$ with notation as in \cite{bourb}.

We construct a $B$-orbit $\mathcal{O}$ for which Theorem \ref{main-intro} fails and prove that Corollary \ref{N-coro} also fails for $X$ (note that $G$ is not simply laced).

\begin{prop}
The closure of the $B$-orbit $\mathcal{O}$ of the element $x=(\scal{e_3,e_1+e_5,e_2+e_6},\scal{e_4,e_5,e_6})$ is not normal.
\end{prop}

\begin{proof}
To prove this result, we describe $B$-orbits $\mathcal{O}$ in $X$ such that the graph $B(X)$ contains the following subgraph (we denote by $P_{\a_1}$ and $P_{\a_2}$ the minimal parabolic subgroups containing $B$ associated to the simple roots $\a_1$ and $\a_2$).

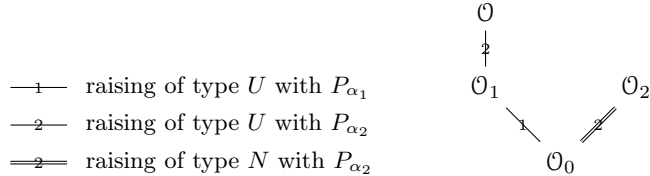
\begin{tabular}{cccc}
\begin{tikzpicture}
  \node (0) at (2,8.2) {};

  \node (1) at (3,7.5)  {};
  \node (2) at (4,7.5) {};
  \node[anchor=west] (3) at (4,7.5)  {\small{raising of type $U$ with $P_{\a_1}$}};

\node (4) at (3,7)  {};
  \node (5) at (4,7) {};
  \node[anchor=west] (6) at (4,7)  {\small{raising of type $U$ with $P_{\a_2}$}};

  \node (7) at (3,6.5)  {};
  \node (8) at (4,6.5) {};
  \node[anchor=west] (9) at (4,6.5)  {\small{raising of type $N$ with $P_{\a_2}$}};

    \draw[double] (7) -- node {\tiny{2}} (8);
    \draw (4) -- node {\tiny{2}} (5);
    \draw (1) -- node {\tiny{1}} (2);
\end{tikzpicture}
&&&
 \begin{tikzpicture}
   \node (Y3) at (9,8)  {$\mathcal{O}$};
  \node (Y1) at (9,7) {$\mathcal{O}_1$};
  \node (Y2) at (11,7)  {$\mathcal{O}_2$};
  \node (Y) at (10,6)  {$\mathcal{O}_0$};

    \draw[double] (Y) -- node {\tiny{2}} (Y2);
    \draw (Y1) -- node {\tiny{2}} (Y3);
    \draw (Y) -- node {\tiny{1}} (Y1);
\end{tikzpicture}
\end{tabular}

\vskip 0.2 cm

\centerline{Subgraph of $\Gamma(X)$}

\medskip

If such a subgraph exists, we claim that the closure of $\mathcal{O}$ is not normal. This was proved in \cite[Corollary 4.4.5]{survey}, we reproduce the simple proof for the convenience of the reader: the morphism $P_{\a_2}\times^B\mathcal{O}_1\to \mathcal{O}$ is birational while its restriction $P_{\a_2}\times^B\mathcal{O}_0\to \mathcal{O}_2$ has non connected fibres. Zariski's Main Theorem gives the conclusion.

We are therefore left to prove that the above graph is indeed a subgraph of $\Gamma(X)$. Note that this will also produce a counterexample to Corollary \ref{N-coro} in the non simply laced case. We define the orbits $\mathcal{O}_0$, $\mathcal{O}_1$ and $\mathcal{O}_2$ as follows:

\begin{tabular}{l}
$\mathcal{O}_0$ is the $B$-orbit of $x_0=(\scal{e_1,e_2+e_4,e_3+e_5},\scal{e_4,e_5,e_6})$\\
$\mathcal{O}_1$ is the $B$-orbit of $x_1=(\scal{e_2,e_1+e_4,e_3+e_6},\scal{e_4,e_5,e_6})$\\
$\mathcal{O}_2$ is the $B$-orbit of $x_2=(\scal{e_1,e_3+e_4,e_2+e_5},\scal{e_4,e_5,e_6})$.\\
\end{tabular}

We first prove the following equalities: $P_{\a_1}x_0=P_{\a_1}x_1$, $P_{\a_2}x_0=P_{\a_2}x_2$ and $P_{\a_2}x_1=P_{\a_2}x$. For this is is enough to produce elements $p_1\in P_{\a_1}$, $p_2\in P_{\a_2}$ and $p\in P_{\a_2}$ such that $p_1x_0=x_1$, $p_2x_0=x_2$ and $px_1=x$. It is enough to take $p_1,p_2,p$ as follows:
$$p_1=\left(\begin{array}{cccccc}
0&1&0&0&0&0\\
1&0&0&0&0&0\\
0&0&1&0&0&0\\
0&0&0&1&0&0\\
0&0&0&0&0&1\\
0&0&0&0&1&0\\
\end{array}\right)\ \ 
p_2=\left(\begin{array}{cccccc}
1&0&0&0&0&0\\
0&1/\sqrt{2}&1/\sqrt{2}&0&0&0\\
0&1/\sqrt{2}&-1/\sqrt{2}&0&0&0\\
0&0&0&-1/\sqrt{2}&1/\sqrt{2}&0\\
0&0&0&1/\sqrt{2}&1/\sqrt{2}&0\\
0&0&0&0&0&1\\
\end{array}\right)$$ 
$$p=\left(\begin{array}{cccccc}
1&0&0&0&0&0\\
0&0&1&0&0&0\\
0&1&0&0&0&0\\
0&0&0&0&1&0\\
0&0&0&1&0&0\\
0&0&0&0&0&1\\
\end{array}\right).$$
Computing the stabiliser of $x_i$ for $i\in\{0,1,2,\emptyset\}$ in $B$, it is easy to compute the dimensions $\dim\mathcal{O}_0=8$, $\dim\mathcal{O}_1=9$, $\dim\mathcal{O}_2=9$ and $\dim\mathcal{O}=10$. Note also that the orbits $\mathcal{O}_1$ and $\mathcal{O}_2$ are distinct: write $x_i=(V_i,W_i)$ for $i\in\{1,2\}$, we have that $V_1$ is in the $B$-orbit of $\scal{e_3,e_5,e_6}$ while $V_2$ is in the $B$-orbit of $\scal{e_1,e_4,e_5}$. This proves that the above graph has the correct shape and we are left to proving that the types of the edges are as above.

To decide if the edge is of type ${\rm U}$, ${\rm T}$ or ${\rm N}$ we use the following criteria (see \cite[Page 405]{springer?} or \cite[Page 268]{brion1}: let $P$ be a minimal parabolic subgroup raising a $B$-orbit $\mathcal{O}$ to a $B$-orbit $\mathcal{O}'$. Let $x\in \mathcal{O}'$ and $P_x$ its stabiliser in $P$. Denote by $S$ the image of $P_x$ in ${\rm Aut}(P/B)$. Then we have:
\begin{itemize}
\item the edge is of type ${\rm U}$ if $S$ contains a positive dimensional unipotent subgroup,
\item the edge is of type ${\rm T}$ if $S$ is a maximal torus in ${\rm Aut}(P/B)$,
\item the edge is of type ${\rm N}$ if $S$ is the normaliser of a maximal torus in ${\rm Aut}(P/B)$.
\end{itemize}
An easy computation of stabiliser proves that the edges are of the above type finishing the proof.
\end{proof}

\end{document}